\documentclass[12pt]{amsart}
\usepackage{amssymb}

\newtheorem{theorem}{Theorem}[section]

\newtheorem{lemma}{Lemma}[section]

\begin{document}

\title[On a conjecture of Pomerance]
{On a conjecture of Pomerance}

\author{L. Hajdu}
\address{\newline L. Hajdu \newline University of Debrecen\\
\newline Institute of Mathematics\\
\newline and the Number Theory Research Group
of the Hungarian Academy of Sciences\\
\newline P.O. Box 12.\\
\newline H-4010 Debrecen\\
\newline Hungary} \email{hajdul@science.unideb.hu}
\thanks{Research supported in part by the OTKA grants K67580
and K75566, and by the T\'AMOP 4.2.1./B-09/1/KONV-2010-0007
project. The project is implemented through the New Hungary
Development Plan, cofinanced by the European Social Fund
and the European Regional Development Fund.}

\author{N. Saradha}
\address{\newline N. Saradha \newline School of Mathematics\\
\newline Tata Institute of Fundamental Research\\
\newline Dr. Homibhabha Road\\ \newline Colaba, Mumbai\\ \newline India}
\email{saradha@math.tifr.res.in}

\author{R. Tijdeman}
\address{\newline R. Tijdeman\newline
Mathematical Institute\newline Leiden
University\newline P.O.Box 9512\newline 2300 RA Leiden\newline The Netherlands}
\email{tijdeman@math.leidenuniv.nl}

\dedicatory{Dedicated to Professor Schinzel on the occasion of his 75th birthday}


\maketitle

\section{Introduction}

Let $k>1$ be an integer. We denote Euler's totient function by $\varphi (k)$ and the number of distinct prime divisors of $k$ by $\omega (k)$. We say that $k$ is a $P-$integer if the first $\varphi (k)$ primes coprime to $k$ form a reduced residue system modulo $k$. In 1980, Pomerance \cite{p} proved the finiteness of the set of $P-$integers. The following conjecture was proposed by him in \cite{p}.

\vskip 2mm

\noindent
{\bf Conjecture of Pomerance.} If $k$ is a $P-$integer, then $k\leq 30$.

\vskip 1mm

\noindent
This conjecture is still open. Recently, Hajdu and Saradha \cite{hs1} and Saradha \cite{s} have given simple conditions under which an integer $k$ is not a $P-$integer. By their results, it follows that

\vskip 1mm

\begin{itemize}
\item
{\it no prime is a $P-$integer except $2$;}
\item
{\it no square or a cube of a prime is a $P-$integer except $4$;}
\item
{\it no integer $k$ with its least prime divisor $>\log k$ is a $P-$integer except when $k\in\{2,4,6\}$.}
\end{itemize}

\vskip 1mm

\noindent
It is easy to check that the only $P-$integers $\leq 30$ are $2,4,6,12,18,30$. It was checked by computation in \cite{hs1} that if $k$ is another $P-$integer, then $k\geq 5.5 \cdot 10^5$. In Theorem \ref{thmcalc} we improve this bound to $10^{11}$. In this paper, we give a quantitative version of the finiteness result of Pomerance and prove the conjecture of Pomerance under the Riemann Hypothesis. We have

\begin{theorem} \label{ub}
If $k$ is a $P-$integer, then $k<10^{3500}$.
\end{theorem}

\begin{theorem} \label{ubrh}
Suppose the Riemann Hypothesis holds. Then the only $P-$integers are $2,4,6,12,18,30$.
\end{theorem}

Theorem \ref{ub} depends on results about the zeros of
the Riemann zeta function. Our method of proof differs from the methods used in \cite{hs1}, \cite{p} and \cite{s}. Our arguments are based on estimates for the number of primes in intervals. We do not use the Jacosthal function and its properties as done in the papers mentioned above.

\section{Lemmas}

Let $p_1<p_2<\dots$ be the increasing sequence of prime numbers. For any $x>1,$ let $\pi(x)$ denote the number of prime numbers not exceeding  $x$, and ${\rm Li}(x) = \lim_{x \to \infty} \int_{t=0}^{1- \epsilon} \frac{dt} {\log t} +
\int_{t=1 + \epsilon}^{x} \frac{dt} {\log t}$. We put $\pi(x)=0$ for $0 \leq x \leq 1$.

\begin{lemma}
\label{pnt} For any $x\in {\mathbb R}$ and $n\in {\mathbb N}$ we have

\noindent
{\rm (i)} $\pi(x) > \frac{x}{\log x}+\frac{x}{\log^2 x}+
\frac {2x}{\log^3 x}$ for $x>88783;$

\noindent
{\rm (ii)} $\pi(x)<\frac{x}{\log x}+\frac{x}{\log^2 x}+
\frac {2.51x} {\log^3 x}$ for $x>355991;$

\noindent
{\rm (iii)} $|\pi(x) - {\rm Li}(x)|< .4394 \frac {x}{(\log
x)^{3/4}} \exp \left(-\sqrt \frac {\log x} {9.646} \right)$ for $x\geq 58;$

\noindent
{\rm (iv)} if the Riemann Hypothesis holds, then
$|\pi(x)-{\rm Li} (x)| < \frac {1} {8 \pi} \sqrt{x} \log x$ for $x>2656;$

\noindent
{\rm (v)} ${\rm Li} (x)> \pi (x)$ for $x \leq 10^{14};$

\noindent
{\rm (vi)} $p_n<n(\log n + \log \log n)$ for $n \geq 6;$

\noindent
{\rm (vii)} $p_n>n\log n$ for $n\geq 1;$

\noindent
{\rm (viii)} $\frac {n} {\varphi(n)} <1.7811\log\log n+\frac {2.51} {\log\log n}$ for $n\geq 3$.

\end{lemma}

\begin{proof} We mention the references where the estimates from Prime Number Theory given in the lemma can be found. \\
(i) Dusart \cite{d3}, p. 2.\\
(ii) Dusart \cite{d1}, p. 40. \\
(iii) Dusart \cite{d1}, p. 41. \\
(iv) Schoenfeld \cite{sc}, p. 339. \\
(v) Kotnik \cite{kot}, p. 59.\\
(vi) Rosser and Schoenfeld \cite{rs}, p. 69. \\
(vii) Rosser and Schoenfeld \cite{rs}, p. 69. \\
(viii) Rosser and Schoenfeld \cite{rs}, p. 72.
\end{proof}

\begin{lemma}
\label{step1} Let $x$ be a real number with $x>712000$. Then we have
$$
2\pi\left(\frac x2\right)-\pi(x)>\frac{.693x}{\log^2 x}.
$$
\end{lemma}

\begin{proof}
We have, by Lemma \ref{pnt}, for $x>712000$,
$$
2\pi(x/2)-\pi(x)>
$$
$$
\frac{x}{\log(x/2)}+\frac{x}{\log^2 (x/2)}+ \frac {2x} {\log^3 (x/2)}-\frac{x}{\log x}-\frac{x}{\log^2 x}-\frac {2.51x} {\log^3 x}>
$$
$$
\frac{x}{\log x \left( 1 - \frac {\log 2} {\log x} \right) } -\frac{x}{\log x} + \frac{x}{\log^2 x  \left( 1 - \frac {\log 2} {\log x} \right)^2 } -\frac{x}{\log^2 x} -\frac {.51x}{\log^3 x}>
$$
$$
\frac {x} {\log x} \cdot \frac {\log 2}{\log x} + \frac {x} {\log^2 x}  \cdot \frac {2\log 2} {\log x} - \frac {.51x}{\log^3 x} > \frac {.693 x}{\log^2 x} .
$$
\end{proof}

\begin{lemma}
\label{gen}
Let $x$ and $y$ be positive real numbers with $x>y$, $x\geq 59$. Then
$$
2\pi (x+y) - \pi (x) - \pi (x+2y)>
$$
$$
\frac {y^2} {(x+2y) \log^2(x+2y)} -  \frac {1.7576(x+2y)} {( \log x)^{3/4}} e^{-\sqrt \frac {\log x} {9.646}}.
$$
\end{lemma}

\begin{proof}
By Lemma \ref{pnt} (iii),
$$
2 \pi (x+y) - \pi (x) - \pi (x+2y) >
$$
$$
2 {\rm Li} (x+y) - {\rm Li} (x) - {\rm Li} (x+2y) -1.7576 \frac {x+2y}{(\log x)^{3/4}} \exp \left(- \sqrt \frac {\log x} {9.646} \right).
$$
Observe that
$$
2 {\rm Li} (x+y) - {\rm Li} (x) - {\rm Li} (x+2y) = \int_x^{x+y}  \frac {dt} {\log t} - \int_{x+y}^{x+2y}  \frac {dt} {\log t}
$$
$$
= \int_x^ {x+y} dt \left( \frac{1}{ \log t} - \frac {1} { \log (t+y)} \right) = \frac {y^2} {\xi \log^2 \xi}
$$
for some $\xi$ with $x < \xi < x+2y $, by the mean value theorem applied twice. Thus
$$
2 \pi (x+y) - \pi (x) - \pi (x+2y) >
$$
$$
\frac {y^2} {(x+2y) \log^2(x+2y)} - 1.7576 \frac {x+2y}{(\log x)^{3/4}} \exp \left(- \sqrt \frac {\log x}{9.646} \right).
$$
\end{proof}

\begin{lemma} \label{schoenfeld}
Suppose the Riemann Hypothesis holds true.\\
Let $x>y>0$, $x \geq 2657$. Then
$$
2 \pi (x+y) - \pi (x) - \pi (x+2y) >
$$
$$
\frac {y^2} {(x+2y) \log^2(x+2y)} - \frac {\log(x+2y)} {\theta} \sqrt{x+2y}
$$
where
$$
\theta=
\begin{cases}
2\pi \ if \ x+2y>10^{14}\\
4\pi \ if \ x+2y \leq 10^{14}.
\end{cases}
$$
\end{lemma}

\begin{proof}
By Lemma \ref {pnt} (iv) and (v),
$$
2\pi (x+y) - \pi (x) - \pi (x+2y) >
$$
$$
2 {\rm Li} (x+y) - {\rm Li} (x) - {\rm Li} (x+2y) - \frac {\log (x+2y)} {\theta}\sqrt{x+2y}.
$$
The lemma follows in the same way as in the proof of Lemma \ref{gen}.
\end{proof}

\section{A criterion for an integer $k$ to be not a $P-$integer}

Suppose $k$ is a $P-$integer $>30$. Let $\varphi(k)+\omega(k)=T$. Then there are exactly $\varphi(k)$ primes belonging to the set $\{p_1,\cdots, p_T\}$ which are coprime to $k$ and form a reduced residue system mod $k$. The remaining $\omega(k)$ primes in this set divide $k$. Let
$$
D_k'=\left\{ i \leq T: p_i\ ({\rm mod}\ k) <\frac{k}{2}\right\},
$$
$$
D_k''=\left\{ i \leq T: p_i\ ({\rm mod}\ k) \geq\frac{k}{2}\right\}
$$
and
$$
D_k'''=\left\{ i \leq T: p_i|k\right\}.
$$
Note that $|D_k'''|=\omega(k)$ where $|A|$ denotes the number of elements of a set $A$. By the symmetry of the residues about $k/2$, we get
$$
|D_k'\setminus D_k'''|=|D_k''\setminus D_k'''|
$$
which implies
\begin{equation}\label{eqn0}
|D_k'|-|D_k''|\leq |D_k'''|=\omega(k).
\end{equation}
Let $t$ be an integer such that $tk <p_T<(t+1)k$. We observe that if $p_T\in (tk,tk+\frac{k}{2})$ we have
$$
|D_k'|= \sum_{n=0}^{t-1}
\left(\pi\left(nk+\frac{k}{2}\right)-\pi(nk)\right)+T-\pi(tk),
$$
$$
|D_k''|= \sum_{n=0}^{t-1}
\left(\pi(nk+k)-\pi\left(nk+\frac{k}{2}\right)\right)
$$
and if $p_T\in \left(tk+\frac{k}{2},tk+k\right)$, then
$$
|D_k'|= \sum_{n=0}^{t}\left( \pi\left(nk+\frac{k}{2}\right)-\pi(nk)\right),
$$
$$
|D_k''|= \sum_{n=0}^{t-1}
\left(\pi(nk+k)-\pi\left(nk+\frac{k}{2}\right)\right)+
T-\pi\left(tk+\frac{k}{2}\right).
$$
Thus we get
$$
|D_k'|-|D_k''|=\sum_{n=0}^{t-1}\left(2\pi\left(nk+\frac{k}{2} \right)-\pi(nk)-\pi(nk+k)\right)+T-\pi(tk)
$$
in the former case, and in the latter case
$$
|D_k'|-|D_k''|=\sum_{n=0}^{t}
\left(2\pi\left(nk+\frac{k}{2}\right)-\pi(nk)
-\pi(nk+k)\right)+\pi(tk+k)-T.
$$
Let $L(k) = t-1$ in the former case and $L(k)=t$ in the latter. Let $L:=L(k)$. We shall use this parameter $L$ later on without any further mentioning. Noting that $T-\pi(tk)$ and $\pi(tk+k)-T$ are both non-negative and that $\omega(k)<\log k$, we find by \eqref{eqn0} the following criterion.

\begin{lemma}\label{criterion}
The integer $k$ is {\it not} a $P-$integer, if
$$
S_L:=\sum_{n=0}^{L}
\left(2\pi\left(nk+\frac{k}{2}\right)-\pi(nk) -\pi(nk+k)\right)>\log k.
$$
\end{lemma}
We note that
$$
tk<p_T\leq p_k\leq k\log(k\log k)
$$
by Lemma \ref{pnt} (vi). Thus
\begin{equation}\label{eqn1}
L \leq t <\log(k\log k).
\end{equation}
On the other hand, using Lemma \ref{pnt} (vii) and (viii), putting $h(k)=1.7811\log\log k+\frac{2.51}{\log\log k}$,
we get
\begin{equation}\label{eqn2}
L+2 \geq t+1>\frac{p_T}{k} \geq  \frac{p_{\varphi(k)}}{k} > \frac{\log k-\log h(k)}{h(k)}.
\end{equation}

\section{A computational result}

\begin{theorem}
\label{thmcalc} If $30<k\leq 10^{11}$, then $k$ is not a $P-$integer. Further, if $k$ is even with $30< k\leq 2\cdot 10^{11}$ then $k$ is not a $P-$integer.
\end{theorem}

\begin{proof} We first prove the statement for $k$ even. In \cite{hs1} it has been computationally verified that no integer $k$ with $30<k<5.5\cdot 10^5$ is a $P-$integer. Hence we may assume henceforth that
$$
5.5\cdot 10^5\leq k\leq 2\cdot 10^{11}.
$$
To cover this interval, we apply a modified version of the algorithm used in \cite{hs1}.

To prove a statement for a given $k$ we apply the following strategy. We find a prime $p>k$ such that $p<p_{\varphi(k)}$ and $p\ ({\rm mod}\ k)$ is also a prime. Then $k$ is not a $P-$integer. To make this strategy work on the whole range for $k$ under consideration, we shall make use of the following two properties. Let $k$ be an integer with $k\geq 5.5\cdot 10^5$. Then we have
\begin{equation}\label{newineq1}
\pi(k+1)+100<\varphi(k)
\end{equation}
and
\begin{equation}\label{newineq2}
p_{\pi(k+1)+100}<1.5k.
\end{equation}
These assertions can be easily checked e.g. by Magma \cite{bcp}, using parts (ii), (vi) and (viii) of Lemma \ref{pnt}.

First we prove the statement for the even values of $k$.
This is done by the algorithm below, which is based on the strategy indicated above.

\vskip 1mm

\noindent
{\bf Initialization.} Let $k_0=5.5\cdot 10^5$. Let $H$ be the list of the first $100$ primes larger than $k_0+1$, i.e. $H=[p_{\pi(k_0+1)+1},\dots,p_{\pi(k_0+1)+100}]$.

\vskip 1mm

\noindent
{\bf Step 1.} Check successively for the primes $p\in H$ whether $p\ ({\rm mod}\ k_0)$ is also a prime. When such a $p$ is found then by \eqref{newineq1}, $k_0$ is not a $P-$integer - proceed to the next step.

\vskip 1mm

\noindent
{\bf Step 2.} Check if $k_0+3$ is a prime. If not, then proceed to Step 3. If so, this is the first element of $H$. Remove this prime from $H$, and append to $H$ the prime $p_{\pi(k_0+1)+101}$ which is the next prime to the last element of $H$.

\vskip 1mm

\noindent
{\bf Step 3.} If $k_0<2\cdot 10^{11}$ then put $k_0:=k_0+2$, and go to Step 1.

\vskip 1mm

Using this procedure, by a Magma program we could check that there is no even $P-$integer in the interval $[5.5\cdot 10^5,2\cdot 10^{11}]$.

Let now $k$ be odd with $5.5\cdot 10^5<k<10^{11}$. Then by our algorithm above, using \eqref{newineq1} and \eqref{newineq2} we know that there exists a prime $p$ satisfying $2k<p<\min \{3k,p_{\varphi(2k)}\}$ such that $q:=p\ ({\rm mod}\ 2k)$ is also a prime. Observe that $q<k$. Thus as $\varphi(k)=\varphi(2k)$, $p$ is a prime such that $k<p<p_{\varphi(k)}$ and $q=p\ ({\rm mod}\ k)$ is also a prime. Hence $k$ is not a $P-$integer and the theorem follows.
\end{proof}

\section{Proofs of the theorems}

\begin{proof}[Proof of Theorem 1.1.] Let $k$ be an integer with $k\geq 10^{3500}$. Then by \eqref{eqn2}, $L>500$. We apply Lemma \ref{pnt} to get
$$
2\pi(k/2) - \pi(k) >
$$
$$
{\frac k {\log(k/2)}}+{\frac k {\log^2(k/2)}}+{\frac {2k}{\log^3(k/2)}}-{\frac k{\log k}}-{\frac k{\log^2 k}}-{\frac {2.51k}{\log^3 k}}.
$$
For $n\geq 1$ we apply Lemma \ref{gen} with $x=nk$, $y=k/2$ to find
$$
2 \pi (nk+k/2) - \pi (nk) - \pi (nk+k)  >
$$
$$
\frac {k} {4(n+1) \log^2(nk+k)} - 1.7576 \frac {nk+k}{(\log
nk)^{3/4}} \exp \left(-\sqrt \frac {\log (nk)} {9.646} \right)
$$
Put
$$
f_0(k):={\frac k {\log \frac k2}}+{\frac k {\log^2\frac k2}}+{\frac {2k}{\log^3 \frac k2}}
-{\frac k{\log k}}-{\frac k{\log^2k}}-{\frac {2.51k}{\log^3k}}
-\log k,
$$
$$
f_n(k):=\frac {k} {4(n+1) \log^2(nk+k)} -   1.7576 \frac {nk+k}{(\log
nk)^{3/4}} \exp \left(-\sqrt \frac {\log (nk)} {9.646} \right)
$$
for $n\geq 1$. A simple calculation shows that
$$
S_L \geq f_0(k)+\sum\limits_{n=1}^L f_n(k)>0
$$
for $L\leq 1500$. This shows that $k$ is not a $P$-integer for such $L$. Hence we may assume that $L>1500$. By \eqref{eqn1} we have $L<\log(k\log k)$. It suffices to show that
$$
f_0(k)+\sum\limits_{n=1}^{1500} f_n(k)+\sum\limits_{n=1501}^L f_n(k)>0.
$$
For this, we first check by Maple that $f_n(k)$ is a strictly monotone decreasing function of $n$. Hence it is enough to show that
$$
f_n(k)+{\frac{f_0(k)+\sum\limits_{i=1}^{1500} f_i(k)}{L-1500}} >
0\ {\rm for}\ n=\log(k\log k)\ {\rm and}\ k =10^{3500}.
$$
We check this again with Maple to get the final contradiction.
\end{proof}

\noindent {\bf Remark.} The constant $9.646$ which occurs in Lemma \ref{pnt}(iii) originates from a zero-free region of the Riemann-zeta function derived by Rosser and Schoenfeld (\cite{rs2} Theorem 1), where the constant appears as $R$. The zero-free region has been widened by Kadiri \cite{kad} where
the corresponding constant $R$ is $5.69693$. If this constant would be substituted into Lemma \ref{pnt} instead of the constant $9.646$ and we follow our argument, we obtain that if $k$ is a $P$-integer, then $k<10^{1000}$. However, we do not know if this substitution is justified.

\begin{proof}[Proof of Theorem 1.2.] Suppose the Riemann Hypothesis is true. Let $k$ be an integer with $k\geq 3\cdot 10^{13}$. By Lemma \ref{step1}, we get
$$
2\pi\left(\frac k2\right) - \pi(k) > \frac{.693k}{\log^2 k} > \log k > \omega (k).
$$
For $n=1,2,\dots,\lfloor\log (k\log k)\rfloor-1$ we apply Lemma \ref{schoenfeld} with $x=nk$, $y=k/2$ to find
$$
2 \pi \left(nk+\frac{k}{2}\right) - \pi (nk) - \pi (nk+k)>
$$
$$
\frac {k}{4(n+1) \log^2(nk+k)} - \frac {\log(nk+k)} {2 \pi}
\sqrt{nk+k}.
$$
The term on the right hand side of the above inequality is positive if
$$
\pi \sqrt{k} > 2(n+1)^{1.5} \log^3(nk+k).
$$
\noindent This is satisfied, since $n<\log (k\log (k))-1$ and
$k\geq 3\cdot 10^{13}$. Hence by Lemma \ref{criterion}, we find that $k$ is not a $P-$integer.

Next we take $k<3\cdot 10^{13}$. By Theorem 4.1, we may assume
$k>10^{11}$. Note that $L<\log(k\log k)\leq 34$.
Further
$$
L<\log k+\log\log k<1.13\log k
$$
giving
$$
k>e^{.88L}>10^{.38L}.
$$
Define
$$
k_L=[10^{\{.38L\}}] 10^{[.38L]}.
$$
where $[x]$ and $\{x\}$ denote the integral and fractional part of any real number $x$. Note that for any fixed $L$ with $L\leq 34$ if $L(k) \geq L$, then $k \in [k_L,3\cdot 10^{13})$. Applying Lemma \ref{schoenfeld} with $x=nk$, $y=k/2$ we find
$$
S_L>2\pi(k/2)-\pi(k)+
$$
$$
+\sum\limits_{n=1}^L\left(\frac {k}{4(n+1) \log^2(nk+k)} - \frac
{\log(nk+k)} {4 \pi} \sqrt{nk+k}\right).
$$
For $n=1,\dots,L$, put
$$
F_n(k):={\frac 1L}\left(\frac{k}{\log(k/2)}+\frac{k}{\log^2(k/2)}+
\frac{2k}{\log^3(k/2)}\right)
$$
$$
-{\frac 1L}\left(\frac{k}{\log k}+
\frac{k}{\log^2 k}+\frac{2.51k}{\log^3 k}+\log k\right)
$$
$$
+\frac {k}{4(n+1) \log^2(nk+k)} - \frac {\log(nk+k)} {4 \pi}
\sqrt{nk+k}.
$$
We have, by Lemma \ref{pnt} (i), (ii),
$$
S_L-\log k>\sum\limits_{n=1}^L F_n(k).
$$
So it is sufficient to show that the right hand side is positive. For this, we proceed as follows. First, let $29\leq L\leq 34$. We calculate the value $k_L$ from its definition above. Thus $(L,k_L)$ is one of the pairs from
$$
\{ (29,10^{11}),(30,2\cdot 10^{11}),(31,6\cdot 10^{11}),(32,10^{12}),(33,3\cdot 10^{12}),(34,8\cdot 10^{12})\}.
$$

We check by Maple that all functions $F_n(k)$ are strictly monotone increasing on $[k_L,3\cdot 10^{13}]$, and further
$$
\sum\limits_{n=1}^L F_n(k_L)>0.
$$
Hence by Lemma \ref{criterion}, there is no $P$-integer
$k$ with $L(k)\in [29,34]$. Now we consider $k \in [10^{11},3\cdot 10^{13}].$ Then obviously $L(k)>0$.
We may assume $1\leq L\leq 28.$ We check that all functions $F_n(k)$ are strictly monotone increasing and the preceding inequality also holds. Hence we conclude that no integer $k\in [10^{11}, 3\cdot 10^{13}]$ is a $P-$integer.
\end{proof}


\begin{thebibliography}{00}
\bibitem {bcp} W. Bosma, J. Cannon, C. Playoust, \textit{ The
Magma algebra system. I. The user language}, J. Symbolic
Comput. \textbf{24} (1997), 235--265.

\bibitem {d1} P. Dusart, \textit{Autour de la fonction qui
compte le nombre de nombres premiers}, Th\`ese, Universit\'e de Limoges, 1998,\ 172 pp.

\bibitem {d2} P. Dusart, \textit{ In\'egaliti\'es explicites
pour $\psi(X)$, $\theta(X)$, $\pi(X)$ et les nombres premiers}, C.R. Math. Rep. Acad. Sci. Canada \textbf{21} (1) (1999), 53--59.

\bibitem{d3} P. Dusart, \textit{Estimates of some functions
over primes without R.H.}, arXiv:1002.0442v1 [math.NT], 2010.

\bibitem {e} P. Erd\H{o}s, \textit{On the integers relatively
prime to $n$ and a number-theoretic function considered by
Jacobsthal}, Math. Scand. \textbf{10} (1962), 163--170.

\bibitem {h} T.R. Hagedorn, \textit{Computation of
Jacobsthal's function $h(n)$ for $n<50$}, Math. Comp. \textbf{78} (2009), 1073--1087.

\bibitem {hs1} L. Hajdu and  N. Saradha, \textit{On a problem
of Recaman and its generalization}, J. Number Theory \textbf{131} (2011), 18--24.

\bibitem{hs2} L. Hajdu and N.Saradha, \textit{Disproof of a
conjecture of Jacobsthal}, Math. Comp., accepted for publication.

\bibitem{kad} H. Kadiri, \textit{Une r\'egion explicite sans
z\'eros pour la fonction $\zeta$ de Riemann}, Acta Arith. \textbf{117} (2005), 303--339.

\bibitem{kot} T. Kotnik, \textit{The prime-counting function
and its analytic approximations}, Adv. Comput. Math. \textbf{29} (2008), 55--70.

\bibitem{mr} J.-P. Massias and G. Robin, \textit{Bornes
effectives pour certaines fonctions concernant les nombres premiers}, J. Th. Nombr. Bordeaux, \textbf{8} (1996), 213--238.

\bibitem {p} C. Pomerance, \textit{A note on the least
prime in an arithmetic progression}, J. Number Theory
\textbf{12} (1980), 218--223.

\bibitem{rs} J.B. Rosser, L. Schoenfeld, \textit{Approximate
formulas for some functions of prime numbers}, Illinois J. Math. \textbf{6} (1962), 64--94.

\bibitem{rs2} J.B. Rosser, L. Schoenfeld, \textit{Sharper
bounds for the Chebyshev functions $\theta (x)$ and $\psi (x)$}, Math. Comput. \textbf{29} (1975), Number 129, 243--269.

\bibitem{sa} N. Saradha, \textit{Conjecture of Pomerance for
some even integers and odd primorials}, Publ. Math. Debrecen, submitted.

\bibitem{sc} L. Schoenfeld, \textit{ Bounds for the Chebyshev
Functions $\theta(x)$ and $\psi(x)$. II"}, Math. Comput. \textbf{30} (1976), Number 134, 337--360.

\bibitem{s} H. Stevens, \textit{On Jacobsthal's $g(n)$
function}, Math. Ann. \textbf{226} (1977), 95--97.

\end{thebibliography}
\end{document}